\documentclass{amsart}
\usepackage{amsfonts}
\usepackage{amssymb}
\usepackage{amsmath}
\usepackage{verbatim}
\usepackage{hyperref}
\usepackage{color}
\usepackage{amsthm}
\usepackage{latexsym}
\usepackage{graphicx}
\usepackage{epstopdf}
\usepackage{float}
\usepackage[utf8]{inputenc}
\usepackage{mathtools}
\usepackage{amsfonts}
\usepackage{tabstackengine}
\usepackage{esint}
\usepackage{verbatim}
\usepackage{mathabx}
\usepackage{bm}

\def\R{\mathbb{R}}
\def\C{\mathbb{C}}

\def\bF{\mathbf{f}}
\def\bp{\mathbf{p}}

\def\one{{\mathbf 1}}

\newtheorem{theorem}{Theorem}[section]

\newtheorem{proposition}[theorem]{Proposition}

\usepackage[utf8]{inputenc}
\begin{document}
\title{Decay Rate of n-Linear Oscillatory Integral Operators in $\R^2$}
\author{Aleksandra Niepla}
\address{Aleksandra Niepla, Department of Mathematics. 310 Malott Hall, Cornell University, Ithaca, New York 14853}
%\curraddr{Department of Mathematics and Statistics,
%Case Western Reserve University, Cleveland, Ohio %43403}
\email{an489@cornell.edu}

\author{Kevin O'Neill}
\address{Kevin O'Neill, Department of Mathematics,
        University of California, 
        Berkeley, CA 94720-3840, USA}
\email{oneill@math.berkeley.edu}

\author{Zhen Zeng}
\address{Zhen Zeng, Department of Mathematics. David Rittenhouse Lab,
209 South 33rd Street,
Philadelphia, PA 19104-6395}
\email{zhenzeng@math.upenn.edu}
\thanks{This material is based upon work supported by the National Science Foundation under Grant Number DMS 1641020.} 

\begin{abstract}
In this paper, we prove $L^p$ decay estimates for multilinear oscillatory integrals in $\R^2$, establishing sharpness through a scaling argument. The result in this paper is a generalization of the previous work by Gressman and Xiao \cite{A.3}.
\end{abstract}
\maketitle
\section{Introduction}

In the groundbreaking work \cite{A.1} of Christ, Li, Tao and Thiele, they initialize the systematic study of general multilinear oscillatory integrals, those of the form

\begin{equation}
    T(f_1,...,f_n)=\int_{\R^m}e^{i\lambda S(x)}\prod_{j=1}^nf_j(\pi_j(x))\phi(x) \ dx,
\end{equation}
where $\lambda\in\R$ is a parameter, $S:\R^m\rightarrow\R$ is a real-valued measurable function, $\phi\in C_0(\R^m)$ is a cutoff function containing the origin in its support, and the $\pi_j$ are orthogonal projections onto proper subspaces $V_j$ of $\R^m$ of common dimension $\kappa<m$. The functions $f_j:V_j\rightarrow \C$ are locally integrable with respect to the Lebesgue measure on $V_{j}$.

A phase $S$ is said to be \textit{degenerate} if there exist functions $S_j:V_j\rightarrow\R$ such that $S(x)=\sum_{j=1}^nS_j\circ\pi_j(x,y)$, and \textit{non-degenerate} otherwise. One may easily see that if $S$ is degenerate, then one does not expect decay in $||T||$ as $\lambda\rightarrow\infty$ since $e^{i\lambda S(x)}$ may be cancelled by appropriate modulations of the functions $f_j$.

Another related concept is called simple nondegeneracy. A function $S$ is said to be {\it simply non-degenerate} if there exists a differential operator $L$ of the form $\prod_{j=1}^n(w_j\circ\nabla)$ with $w_j\in V_j^\perp$ such that $L(S)\nequiv0$. If $S$ is simply nondegenerate, then it is nondegenerate, though the converse generally fails. However, the two properties are equivalent when the subspaces $V_j$ are of codimension 1, as will be the case for the remainder of this paper. See \cite{A.1} for reference.

%However, due to the complexity of this problem, they only consider the functions lying in $L^{\infty}$ space. Though great progress has been made, the characterization of the decay rate with respect to different exponents is beyond the scope of their work.

Recently, Gressman and Xiao \cite{A.3} proved the following theorem on the decay rate of a multilinear oscillatory integral operator.

\begin{theorem}\label{thm:GX}
Let $S(x,y)$ be a real-analytic function defined in a neighborhood of $0\in\R^2$. Assume that
\begin{equation}
    |\partial_x\partial_y(\partial_x-\partial_y)S(x,y)|\geq1
\end{equation}
for all $(x,y)$ in the convex hull of the support of $\phi$. Then, there exists $C>0$ such that
\begin{equation}
    \left|\iint e^{i\lambda S(x,y)}f(x)g(y)h(x+y)\phi(x,y)dxdy\right|\leq|\lambda|^{-1/4} C||f||_p||g||_q||h||_r
\end{equation}
whenever $p,q,r\in[2,4)$ and $p^{-1}+q^{-1}+r^{-1}=5/4$.
\end{theorem}

Theorem \ref{thm:GX} is an improvement of a result of Li \cite{A.2}, which establishes a decay rate of $|\lambda|^{-1/6}$ when $p=q=r=2$. The result of Li may be obtained by interpolating Theorem \ref{thm:GX} with the trivial bounds given by Young's convolution inequality. Both proofs use a $TT^*$ argument to reduce to a theorem of H\"ormander on bilinear oscillatory integrals \cite{Hor}. Also, both decay rates are sharp for the given set of exponents. 

Another generalization of Li's work in higher dimensions appears in concurrent work by the authors \cite{US}. The work of Christ and Oliveira e Silva \cite{CD} also addresses decay properties in higher dimensions.

In this paper, we generalize Theorem \ref{thm:GX} to higher levels of multilinearity. % the multilinear oscillatory integral we are dealing with is of the form
Denote $\bF=(f_1,...,f_n)$ and $\bp=(p_1,...,p_n)$. Let $a_j\in\R^2\setminus\{0\}$ for $1\leq j\leq n$ and define an $n$-linear operator $\lambda\geq1$ by

\begin{equation}
\Lambda_n(\bF):=\int_{\R^2}e^{i\lambda S(x,y)}\prod_{j=1}^n f_j(a_j\cdot (x,y))\phi(x,y)dxdy.
\end{equation}

We say $a_1,...,a_n$ lie in general position if no $a_i$ is a scalar multiple of $a_j$ for $i\neq j$. The definition of general position is quite natural since if there exist $i\neq j$ such that $a_i$ is a scalar multiple of $a_j$, then by combining $f_i(a_i\cdot(x,y))$ and $f_j(a_j\cdot (x,y))$, we can reduce the problem to lower multilinearity. Note that by a change of variables, Theorem \ref{thm:GX} holds for any three rank vectors in general position in place of $(0,1), (1,0)$, and $(1,1)$. (Furthermore, if one allows for a function $g(x,y)$ of two variables, no decay is guaranteed since one may take $g(x,y)=e^{-i\lambda S(x,y)}$. Thus, our work covers all nontrivial multilinear oscillatory integrals in $\R^2$.)

Below is our main result:

%Let $\bF=(f_1,...,f_n)$ and $\bp=(p_1,p_2,p_3)$. Let $a_j\in\R^2$ for $1\leq j\leq n$ and define the following $n$-linear operator with parameter $\lambda\geq1$:

%\begin{equation}
%\Lambda_n(\bF)=\int_{\R^2}e^{i\lambda S(x,y)}\prod_{j=1}^n f_j(a_j\cdot (x,y))\phi(x,y)dxdy.
%\end{equation}

%We say the $a_1,...,a_n$ lie in general position if no $a_i$ is a scalar multiple of $a_j$ for $i\neq j$.

\begin{theorem}\label{theorem: higher levels}
Let $n\geq4$ and let $a_j=(b_j,c_j)\in \R^2\setminus\{0\}$ ($1\leq j\leq n$) lie in general position and define $D_n=\prod_{j=1}^n(c_j\partial_x-b_j\partial_y)$. Then, there exists $C>0$, depending only on $\phi$ and the $a_j$, such that if $|D_nS(x,y)|\geq1$ for all $(x,y)$ in the convex hull of the support of $\phi$,
\begin{equation}\label{eq: higher levels}
|\Lambda_n(\bF)|\leq C|\lambda|^{-2^{1-n}}\prod_{j=1}^n||f_j||_{p_j}
\end{equation}
for \begin{equation}\label{eq: higher exponents}
\bp=\left(2,2,\frac{2^{n-1}}{2^{n-2}-1},...,\frac{2^{n-1}}{7},\frac{2^{n-1}}{3}+\epsilon,2^{n-1}-\frac{9\cdot2^{n-3}\epsilon}{2^{n-3}+3\epsilon}\right)
\end{equation}
and $0<\epsilon<1$.
\end{theorem}

Bounds for other choices of $\bp$ may be obtained from interpolation.

In section \ref{sec:proof}, we will prove Theorem \ref{theorem: higher levels}. The proof is by induction with the base case $n=4$. The proof of the base case will in fact use Theorem \ref{thm:GX} and in many ways mimic the induction step. For instance, both parts use a $TT^*$ argument to reduce the degree of linearity by 1. However, there are subtle differences in the choices of exponents and the finishing steps, hence we present the two parts of the proof separately.

In Section \ref{sec:scaling}, we will show that the bounds from Theorem \ref{theorem: higher levels} are sharp through a scaling argument.

The authors would like to thank the organizers of the AMS Mathematical Re-search Committee on Recent Developments in Oscillatory Integrals, in particular, Philip Gressman, for their support of this project.

\section{Proof of the Main Result}\label{sec:proof}

\begin{proof}
Base case ($n=4$): By a change of coordinates, suppose that $a_4=(0,1)$. Writing $a_j=(b_j,c_j)$ for $b_j,c_j\in\R$, the general position hypothesis implies $b_j\neq0$ for $j=1,2,3.$ Begin by observing that
\begin{equation}
|\Lambda_4(\bF)|\leq||T(f_1,f_2,f_3)||_2||f_4||_2,
\end{equation}
where
\begin{equation}
T(f_1,f_2,f_3)(y)=\int e^{i\lambda S(x,y)}\prod_{j=1}^3f_j(b_jx+c_jy)\phi(x,y)dx.
\end{equation}
Applying the $TT^*$ method, we have
\begin{multline}
||T(f_1,f_2,f_3)||_2^2=\iiint e^{i\lambda(S(x_1,y)-S(x_2,y))}\prod_{j=1}^3f_j(b_jx_1+c_jy)\overline{f_j}(b_jx_2+c_jy)\\
\times\phi(x_1,y)\phi(x_2,y)dx_1dx_2dy.
\end{multline}

Making the change of variables $x=x_1, y=y, \tau=x_2-x_1$,
\begin{equation}
||T(f_1,f_2,f_3)||_2^2=\int\left[\iint e^{i\lambda S_\tau(x,y)}\prod_{j=1}^3F_{j,\tau}(a_j\cdot(x,y))\phi_\tau(x,y)\right]d\tau,
\end{equation}
where $S_\tau(x,y)=S(x,y)-S(x+\tau,y)$, $F_{j,\tau}(z)=f_j(z)\overline{f}(z+b_j\tau)$, and $\phi_\tau(x,y)=\phi(x,y)\phi(x+\tau,y)$.

%\int\left[\iint e^{i\lambda(S(x,y)-S(x,y+\tau))}\prod_{j=1}^3f_j(b_jx+c_jy)\overline{f}(b_jx+c_jy_2+c_j\tau)\phi(x,y)\phi(x,y+\tau)dxdy_1\right]d\tau\\

Observe that by the Fundamental Theorem of Calculus,
\begin{equation}\label{eq:FTC}
S_\tau(x,y)=S(x,y)-S(x+\tau,y)=-\int_0^\tau\partial_xS(x+t,y)dt,
\end{equation}
hence
\begin{equation}
\left(\prod_{j=1}^3D_j\right)S_\tau(x,y)=-\int_0^\tau D_4S(x+t,y)dt,
\end{equation}
where $D_j=c_j\partial_x-b_j\partial_y$.

By the assumption that $|D_4S(x,y)|\geq1$ for all $(x,y)$ in the convex hull of the support of $\phi$ and the Mean Value Theorem, $\left|\left(\prod_{j=1}^3D_j\right)S_\tau(x,y)\right|\geq\tau$.

Applying Theorem \ref{thm:GX} gives

\begin{equation}
||T(f_1,f_2,f_3)||_2^2\leq C\int |\tau\lambda|^{-1/4}||F_{1,\tau}||_2||F_{2,\tau}||_{2+\delta}||F_{3,\tau}||_{\frac{8+4\delta}{2+3\delta}}d\tau
\end{equation}
for arbitrarily small $\delta>0$. Applying multilinear H\"older's inequality, we obtain
\begin{multline}\label{eq:post Holder}
||T(f_1,f_2,f_3)||_2^2\leq C|\lambda|^{-1/4}\left(\int ||F_{1,\tau}||_2^2d\tau\right)^{1/2}\\
\times\left(\int |\tau|^{\frac{-2-\delta}{4}}||F_{2,\tau}||_{2+\delta}^{2+\delta}d\tau\right)^{1/(2+\delta)}\left(\int ||F_{3,\tau}||_{\frac{8+4\delta}{2+3\delta}}^{\frac{4+2\delta}{\delta}}d\tau\right)^{\frac{\delta}{4+2\delta}}
\end{multline}

We now address each of the above three terms separately. First,
\begin{align*}
\int ||F_{1,\tau}||_2^2d\tau&=\int\int |f_1(z)|^2|f_1(z+b_j\tau)|^2dzd\tau\\
&\leq C\left|\left|f_1\right|\right|_2^4.
\end{align*}

To bound the second term, we use the Hardy-Littlewood-Sobolev inequality, obtaining
\begin{align*}
\int |\tau|^{\frac{-2-\delta}{4}}||F_{2,\tau}||_{2+\delta}^{2+\delta}d\tau&=\int\int |\tau|^{\frac{-2-\delta}{4}}|f_2(z)|^{2+\delta}|f_2(z+b_j\tau)|^{2+\delta}dzd\tau\\
&\leq C|||f_2|^{2+\delta}||_{\frac{8}{6-\delta}}^2\\
&=C||f_2||_{\frac{16+8\delta}{6-\delta}}^{4+2\delta}
\end{align*}

Lastly, for functions $g,h$, define 

\begin{equation}
g*'h(z):=\int g(x)h(x+b_3z)dx.
\end{equation}

We treat the third term through Young's convolution inequality (for $*'$ in place of $*$) and the elementary equality $||f^s||_{s'}=||f||_{ss'}^s$ for exponents $0<s,s'<\infty$. Denote exponents $p=\frac{8+4\delta}{2+3\delta}$ and $q=\frac{4+2\delta}{\delta}$. Then,

\begin{align*}
\left(\int ||F_{3,\tau}||_{\frac{8+4\delta}{2+3\delta}}^{\frac{4+2\delta}{\delta}}d\tau\right)^{\frac{\delta}{4+2\delta}}&=\left|\left|\left(\int f_3(x)^p\overline{f_3}(x+b_3\tau)^p dx\right)^{1/p}\right|\right|_{L^q(\tau)}\\
&=||(f_3^p*'\overline{f_3^p})^{1/p}||_q\\
&=||f_3^p*'\overline{f_3^p}||_{q/p}^{1/p}\\
&\leq||f_3^p||_r^{2/p}=||f_3||_{pr}^2,
\end{align*}
where $r$ is defined by the relation $\frac{2}{r}=1+\frac{p}{q}$. Hence, $\frac{2}{pr}=\frac{1}{p}+\frac{1}{q}$ and $pr=\frac{16+8\delta}{2+5\delta}$.

Returning to \eqref{eq:post Holder} and applying the bounds for the individual integrals gives
\begin{equation}
||T(f_1,f_2,f_3)||_2^2\leq C|\lambda|^{-1/4}||f_1||_2^2||f_2||_{\frac{16+8\delta}{6-\delta}}^2||f_3||_{\frac{16+8\delta}{2+5\delta}}^2,
\end{equation}
resulting in the conclusion
\begin{equation}
|\Lambda_4(\bF)|\leq|\lambda|^{-1/8}||f_1||_2||f_2||_{\frac{16+8\delta}{6-\delta}}||f_3||_{\frac{16+8\delta}{2+5\delta}}||f_4||_2.
\end{equation}

Let $r_1=\frac{16+8\delta}{6-\delta}$ and $r_2=\frac{16+8\delta}{2+5\delta}$. Since $\delta>0$ is a small parameter, $r_1$ is a small amount greater than $8/3$ and $r_2$ is a small amount less than 8 such that $\frac{1}{r_1}+\frac{1}{r_2}=\frac{1}{2}$. Thus, letting $\epsilon=r_1-8/3$, one recovers the exponents of $8/3+\epsilon$ and $8-\frac{18\epsilon}{2+3\epsilon}$ given in the statement of Theorem \ref{theorem: higher levels}.\\

Induction step: Let $\epsilon>0$ and suppose Theorem \ref{theorem: higher levels} holds for a particular $n$. By a change of coordinates, suppose that $a_{n+1}=(0,1)$. Again, write $a_j=(b_j,c_j)$ for $b_j,c_j\in\R$, and the general position hypothesis implies $b_j\neq0$ for $1\leq j\leq n$. As in the proof of the base case, one may show that
\begin{equation}\label{eq:fromsim}
|\Lambda_{n+1}(\bF)|^2\leq||f_{n+1}||_2^2\int\left[\iint e^{i\lambda S_\tau(x,y)}\prod_{j=1}^nF_{j,\tau}(a_j\cdot(x,y))\phi_\tau(x,y)dxdy\right]d\tau,
\end{equation}
where $S_\tau(x,y)=S(x,y)-S(x,y+\tau)$, $F_{j,\tau}(z)=f_j(z)\overline{f}(z+b_j\tau)$, and $\phi_\tau(x,y)=\phi(x,y)\phi(x,y+\tau)$.

Again using \eqref{eq:FTC},
\begin{equation}
\left(\prod_{j=1}^nD_j\right)S_\tau(x,y)=-\int_0^\tau D_{n+1}S(x+t,y)dt.
\end{equation}
Thus, by the assumption that $|D_{n+1}S(x,y)|\geq1$ for all $(x,y)\in \text{supp }\phi$ and the Mean Value Theorem, 
\begin{equation}\label{eq:theabove}
\left|\left(\prod_{j=1}^nD_j\right)S_\tau(x,y)\right|\geq\tau.
\end{equation}

By the induction hypothesis and \eqref{eq:theabove}, the integral on the right hand side of \eqref{eq:fromsim} is bounded above by

\begin{equation}\label{eq:tau int}
\int C|\lambda\tau|^{-2^{1-n}}\prod_{j=1}^n||F_{j,\tau}||_{L^{p_j}} d\tau,
\end{equation}
where 

\begin{equation}
\bp=\left(2,2,\frac{2^{n-1}}{2^{n-2}-1},...,\frac{2^{n-1}}{7},\frac{2^{n-1}}{3}+\frac{\epsilon}{2},2^{n-1}-\frac{9\cdot2^{n-3}(\epsilon/2)}{2^{n-3}+3(\epsilon/2)}\right).
\end{equation}

Applying H\"older's inequality, the integral in \eqref{eq:tau int} is bounded by

\begin{multline}
C|\lambda|^{-2^{1-n}}\left(\int ||F_{1,\tau}||_2^2d\tau\right)^{1/2}\left(\int |\tau|^{-2^{2-n}}||F_{2,\tau}||_2^2d\tau\right)^{1/2}\\
\times\sup_\tau||F_{3,\tau}||_{p_3}\cdots\sup_\tau||F_{n,\tau}||_{p_n}.
\end{multline}

As before, we have

\begin{equation}
\left(\int ||F_{1,\tau}||_2^2d\tau\right)^{1/2}\leq||f_1||_2^2.
\end{equation}

Similar to before, we bound the second term using the Hardy-Littlewood-Sobolev inequality.

\begin{align*}
\int |\tau|^{-2^{2-n}}||F_{2,\tau}||_2^2d\tau&=\int\int |\tau|^{-2^{2-n}}|f_2(z)|^2|f_2(z+b_j\tau)|^2dzd\tau\\
&\leq C||f_2^2||_{\frac{2^{n-1}}{2^{n-1}-1}}^2\\
&=C||f_2||_{\frac{2^n}{2^{n-1}-1}}^4.
\end{align*}

Lastly, for all $3\leq j\leq n$, we see that

\begin{equation}
\sup_\tau||F_{j,\tau}||_{p_j}=\sup_\tau\left|\int f_j(x)^{p_j}\overline{f_j}(x+b_j\tau)^{p_j}dx\right|^{1/p_j}=||f_j||_{2p_j}^2.
\end{equation}

Substituting these bounds into \eqref{eq:fromsim}, we see that

\begin{equation}
|\Lambda_{n+1}(\bF)|^2\leq C|\lambda|^{-2^{1-n}}||f_{n+1}||_2^2||f_1||_2^2||f_2||_{\frac{2^n}{2^{n-1}-1}}^2||f_3||_{2p_3}^2\cdots||f_n||_{2p_n}^2.
\end{equation}
Taking the square root of both sides and checking the exponents match their desired values completes the proof.

\end{proof}

\section{Sharpness by Scaling Argument}\label{sec:scaling}

In this section, we prove that Theorem \ref{theorem: higher levels} is sharp with respect to the power of $\lambda$ in \eqref{eq: higher levels} and for the exponents given by \eqref{eq: higher exponents} and their interpolations.

\begin{proposition}
Let $1\leq p_j\leq \infty$ where $1\leq j\leq n$ and $\sum_j p_j^{-1}=\frac{2^n-n}{2^{n-1}}$. Fix $\phi\in C_0^\infty(\R^2)$ and $a_j\in\R^2$ lying in general position.

Then, for all phases $S(x,y)$ which are real-analytic in a neighborhood of 0, there exists $C>0$ such that the functional $\Lambda_n:L^{p_1}\times\cdots\times L^{p_n}\rightarrow\C$ has operator norm $||\Lambda_n||\geq C|\lambda|^{\frac{-1}{2^{n-1}}}$.
\end{proposition}

\begin{proof}
Any polynomial in $x$ and $y$ of degree less than or equal to $n-1$ vanishes under $D_n$ and is degenerate; hence its contribution to $e^{i\lambda S(x,y)}$ may be absorbed by the $f_j$. Thus, we may take $S(x,y)=S_n(x,y)+\tilde{S}(x,y)$, where $S_n$ is a homogeneous polynomial of degree $n$ and $\tilde{S}(x,y)$ has Taylor expansion at the origin consisting of terms of only degree $n+1$ or higher.

For $\lambda\geq1$, pick $f_j=\one_{[0,\lambda^{-1/n}]}$. Then,

\begin{equation}
\Lambda_n(\bF)=\int_{\R^2}e^{i\lambda S(x,y)}\prod_{j=1}^n \one_{[0,\lambda^{-1/n}]}(a_j\cdot (x,y))\phi(x,y)dxdy.
\end{equation}

By the homogeneity of $S_n(x,y)$, the properties of $\tilde{S}(x,y)$, and the change of variables $x=\lambda^{-1/n}x', y=\lambda^{-1/n}y'$, we have
\begin{multline}
\Lambda_n(\bF)=\lambda^{-2/n}\int_{\R^2}e^{iS_n(x',y')+O(\lambda^{-1})}\prod_{j=1}^n \one_{[0,1]}(a_j\cdot (x',y'))\\
\times\phi(\lambda^{-1/n}x',\lambda^{-1/n}y')dx'dy'.
\end{multline}

The only instances of dependence on $\lambda$ in the above integral appear in the $O(\lambda^{-1})$ term, which goes to 0 as $\lambda\rightarrow\infty$, and in the term $\phi(\lambda^{-1/n}x',\lambda^{-1/n}y')$, which converges to $\phi(0)$ uniformly on the support of $\prod_{j=1}^n \one_{[0,1]}(a_j\cdot (x',y'))$. Thus, $\Lambda_n(\bF)\geq C|\lambda|^{-2/n}$ for some $C>0$.

For each $j$, we see that $||f_j||_{p_j}=\lambda^{\frac{-1}{np_j}}$, hence
\begin{equation}
\prod_j||f_j||_{p_j}=\lambda^{-\left(\frac{2^n-n}{n2^{n-1}}\right)}.
\end{equation}

Therefore,
\begin{equation}
||\Lambda_n||\geq\frac{|\Lambda_n(\bF)|}{\prod_j||f_j||_{p_j}}\geq \frac{C\lambda^{-2/n}}{\lambda^{-\left(\frac{2^n-n}{n2^{n-1}}\right)}}=\lambda^{\frac{2^n-n-2^n}{n2^{n-1}}}=\lambda^{\frac{-1}{2^{n-1}}}.
\end{equation}

\end{proof}


\begin{thebibliography}{}



\bibitem{A.1} M. Christ, X. Li, T. Tao, C. Thiele, On multilinear oscillatory integrals, nonsingular and singular, \textit{Duke Math. J.}, 130 (2) (2005) 321-351.

\bibitem{CD} M. Christ, D. Oliveira e Silva, On trilinear oscillatory integrals, \textit{Revista Matem{\'{a}}tica Iberoamericana}, 30 (2) (2014) 667-684.

\bibitem{A.3} P. Gressman, L. Xiao, Maximal decay inequalities for trilinear oscillatory integrals of convolution type \textit{Journal of Functional Analysis} 271 (2016) 3695-3726. 

\bibitem{Hor} L. H\"{o}rmander, Oscillatory integrals and multipliers on $FL^{p}$, \textit{Ark. Mat.}, 11: (1973) 1-11. 

\bibitem{A.2} X. Li, Bilinear Hilbert transforms along curves, I: the monomial case, \textit{Anal. PDE} 6 (1) (2013) 197-220. 

\bibitem{US} A. Niepla, K. O'Neill, Z. Zeng, Estimating Oscillatory Integrals of Convolution Type in $\R^d$. \textit{In preparation.}


% \bibitem{CGP} T. C. Collins, A. Greenleaf and M. Pramanik, A multidimensional resolution of singularities with application to analysis, \textit{Amer. J. Math.}, 135: (2013) 1179-1252.

\end{thebibliography}
\end{document}